\documentclass[12pt]{article}
\usepackage{amsmath,amsthm,amssymb,amsfonts,graphicx,wrapfig}
\usepackage[english]{babel}
\usepackage{indentfirst}
\usepackage{xcolor}

\newtheorem{theorem}{Theorem}
\newtheorem{lemma}[theorem]{Lemma}

\date{}
\begin{document}
	
\newcounter{casenum}
\newenvironment{caseof}{\setcounter{casenum}{1}}{\vskip.5\baselineskip}
\newcommand{\case}[2]{\vskip.5\baselineskip\par\noindent {\bfseries Case \arabic{casenum}:} #1\\#2\addtocounter{casenum}{1}}
	
\title{EMSO(FO$^2$) 0-1 law fails for all dense random graphs}

\author{M. Akhmejanova\footnote{Moscow Institute of Physics and Technology, Laboratory of Combinatorial and Geometric structures, Dolgoprudny, Russia}, M. Zhukovskii\footnotemark[\value{footnote}] \footnote{Adyghe State University, Caucasus mathematical center, Maykop, Republic of Adygea, Russia} \footnote{The Russian Presidential Academy of National Economy and Public Administration, Moscow, Russia}}

%\footnotetext[1]{Moscow Institute of Physics and Technology, Laboratory of Combinatorial and Geometric structures, Dolgoprudny, Russia; \; mechmathrita@gmail.com}

\maketitle

\begin{center}
{\bf Abstract}
\end{center}

In this paper, we disprove EMSO(FO$^2$) convergence law for the binomial random graph $G(n,p)$ for any constant probability $p$. More specifically, we prove that there exists an existential monadic second order sentence with 2 first order variables such that, for every $p\in(0,1)$, the probability that it is true on $G(n,p)$ does not converge.

\vspace{0.5cm}

\section{Introduction}
\label{sec:intro}

For undirected graphs, {\it sentences in the monadic second order logic} (MSO sentences) are constructed using relational symbols $\sim$ (interpreted as adjacency) and $=$, logical connectives $\neg,\rightarrow,\leftrightarrow,\vee,\wedge$, first order (FO) variables $x,y,x_1, \ldots$ that express vertices of a graph, MSO variables $X,Y,X_1,\ldots$ that express unary predicates, quantifiers $\forall,\exists$ and parentheses (for formal definitions, see~\cite{Libkin}). If, in an MSO sentence $\phi$, all the MSO variables are existential and in the beginning, then the sentence is called {\it existential} monadic second order (EMSO). For example, the EMSO sentence
$$
 \exists X \quad[\exists x_1\exists x_2\,\,X(x_1)\wedge \neg X(x_2)]\wedge\neg[\exists y \exists z\,\,X(y)\wedge\neg X(z)\wedge y\sim z]
$$
expresses the property of being disconnected. Note that this sentence has 1 monadic variable and 4 FO variables but it can be easily rewritten with only 2 FO variables by identifying $y$ with $x_1$ and $z$ with $x_2$. In what follows, for a sentence $\phi$, we use the usual notation from model theory $G\models\phi$ if $\phi$ is true for $G$.\\

In~\cite{Kaufmann_Shelah}, Kaufmann and Shelah disproved the MSO 0-1 law ({\it 0-1 law for a logic $\mathcal{L}$} states that every sentence $\varphi\in\mathcal{L}$ is either true on (asymptotically) almost all graphs on the vertex set $[n]:=\{1,\ldots,n\}$ as $n\to\infty$, or false on almost all graphs). Moreover, they even disproved a weaker logical law which is called the MSO convergence law ({\it convergence law for a logic $\mathcal{L}$} states that, for every sentence $\varphi\in\mathcal{L}$, the fraction of graphs on the vertex set $[n]$ satisfying $\varphi$ converges as $n\to\infty$). In terms of random graphs, their result can be formulated as follows: there exists an MSO sentence $\varphi$ such that ${\sf P}(G(n,1/2)\models\varphi)$ does not converge as $n\to\infty$. Recall that, for $p\in(0,1)$, the binomial random graph $G(n,p)$ is a graph on $[n]$ with each pair of vertices connected by an edge with probability $p$ and independently of other pairs. For more information, we refer readers to the books \cite{AS,Bollobas,Janson}. In contrast, $G(n,1/2)$ obeys first-order (FO) 0-1 law~\cite{Fagin,Glebskii}. In 2001, Le Bars~\cite{Le_Bars} disproved EMSO convergence law for $G(n,1/2)$ and conjectured that, for EMSO sentences with 2 FO variables (or, shortly, EMSO(FO$^2$) sentences), $G(n,1/2)$ obeys the zero-one law. In 2019, Popova and the second author~\cite{Popova} disproved this conjecture. Notice that all the above results but the last one can be easily generalized to arbitrary constant edge probability $p$. In~\cite{Popova}, it is noticed that the Le Bars conjecture fails for a dense set of $p\in(0,1)$. In this paper, we disprove the Le Bars conjecture for all $p\in(0,1)$. We prove something even stronger: there exists a EMSO(FO$^2$) sentence $\varphi$ such that, for every $p\in(0,1)$, $\{\mathbf{P}[G(n, p) \models \varphi]\}_{n}$ does not converge. Notice that this one sentence disproves the convergence law for all $p$. Let us define the sentence.\\

%In 1969, Y.V. Glebskii, D.I. Kogan, M.I. Liogon'kii and V.A. Talanov, and independently R. Fagin in 1976~\cite{Fagin}, proved that any FO sentence is either true for almost all graphs, or false for almost all graphs. Clearly, this result can be reformulated in terms of the binomial random graph $G(n,1/2)$. For arbitrary $p$, $G(n,p) = (V_n,E)$, where $V_n=\{1,\ldots,n\}$, and each pair of vertices is connected by an edge with probability $p$ and independently of other pairs. For more information, we refer readers to the books \cite{AS,Bollobas,Janson}. {\it The zero-one law} of Glebskii et al. and Fagin says that, for every FO sentence $\phi$, either ${\sf P}(G(n,1/2)\models\phi)\to 0$ as $n\to\infty$, or ${\sf P}(G(n,1/2)\models\phi)\to 1$ as $n\to\infty$ (or, in other words, either $\phi$ is true for $G(n,1/2)$ a.a.s., or false a.a.s.). Below, we give a brief history of studying logical laws for this random graph model, for more details (especially, for FO logic) see, e.g.,~\cite{Survey,Strange}. 

Let $X(k, \ell, m)$ be the number of 6-tuples $\left(X_{1},x_1,X_{2},x_2,X_3,x_3\right)$, consisting of sets $X_1,X_2,X_3\subset[n]$ and vertices $x_1\in X_1$, $x_2\in X_2$, $x_3\in X_3$, such that
\begin{itemize}
\item $|X_1|=k,|X_2|=\ell,|X_3|=m$ and $X_i\cap X_j=\varnothing$ for $i\neq j$,
\item each $X_i$ dominates $[n]\setminus{(X_{1} \sqcup X_{2}\sqcup X_3)}$, i.e. every vertex from $[n]\setminus{(X_{1} \sqcup X_{2}\sqcup X_3)}$ has at least one neighbor in each $X_i$,
%\item $x_i\in X_i$,
\item for any distinct $i,j\in\{1,2,3\}$, there is exactly one edge between $X_i$ and $X_j$ --- namely, the edge between $x_i$ and $x_j$.
\end{itemize}

\bigskip
\begin{theorem}\label{mnt:theorem}
For any constant $p\in(0,1)$,  ${\sf P}(\exists k,\ell,m\,\,X(k,\ell,m)>0)$ does not converge as $n\to\infty$.
\end{theorem}

Clearly, the property $\{\exists k,\ell,m\,\,X(k,\ell,m)>0\}$ can be defined in EMSO(FO$^2$), e.g., by the following sentence:
$$
 \exists X_1\exists X_2\exists X_3\quad\mathrm{DIS}(X_1,X_2,X_3)\wedge\mathrm{DOM}(X_1,X_2,X_3)\wedge\phi_1(X_1,X_2,X_3)\wedge\phi_2(X_1,X_2,X_3),
$$
where the formula
$$
 \mathrm{DIS}(X_1,X_2,X_3)=\bigwedge_{1\leq i<j\leq 3}\biggr(\forall x \forall y\quad [X_i(x)\wedge X_j(y)]\Rightarrow [x\neq y]\biggl)
$$
says that $X_1,X_2,X_3$ are disjoint; the formula
$$
 \mathrm{DOM}(X_1,X_2,X_3)=\forall x\quad \biggl[\neg(X_1(x)\vee X_2(x)\vee X_3(x))\biggr]\Rightarrow\biggl[\bigwedge_{j=1}^3(\exists y\,\,\, X_j(y)\wedge (x\sim y))\biggr]
$$
says that each vertex from $[n]\setminus(X_1\sqcup X_2\sqcup X_3)$ has a neighbor in each $X_i$; the formula
$$
 \phi_1(X_1,X_2,X_3)=\bigwedge_{i=1}^3\biggl(\exists x\quad X_i(x)\wedge\biggl(\forall y\,\,\, [(y\neq x)\wedge X_i(y)]\Rightarrow\biggl[\forall x\,\,\,\biggl(\bigvee_{j\neq i} X_j(x)\biggr)\Rightarrow (x\nsim y)\biggr]\biggr)\biggr)
$$
says that, for every $i\in\{1,2,3\}$, there is at most one vertex that has neighbors in sets $X_j$, $j\neq i$; the formula
$$
 \phi_2(X_1,X_2,X_3)=\bigwedge_{1\leq i<j\leq 3}\biggl(\exists x\exists y\quad X_i(x)\wedge X_j(y)\wedge(x\sim y)\biggr)
$$
says that, for any two distinct $X_i,X_j$, there exists an edge between them. Clearly, $\phi_1\wedge\phi_2$ is true if and only if there exist $x_1\in X_1$, $x_2\in X_2$, $x_3\in X_3$, such that, for any distinct $i,j\in\{1,2,3\}$, there is exactly one edge between $X_i$ and $X_j$ --- the edge between $x_i$ and $x_j$.\\

We prove Theorem~\ref{mnt:theorem} in the following way. First, we show that, for some sequence of positive integers $(n^{(1)}_i,\,i\in\mathbb{N})$, $\sum_{k,\ell,m}{\sf E}X(k,\ell,m)\to 0$ (random variables are defined on $G(n^{(1)}_i,p)$) as $i\to\infty$. Then, we show that, for another sequence $(n^{(2)}_i,\,i\in\mathbb{N})$, there exists $k=k(i)$ such that ${\sf P}(X(k,k,k)>0)$ is bounded away from $0$ for all large enough $i$ (using second moment methods). 

We compute ${\sf E}X(k,\ell,m)$ and study its behavior in Section~\ref{expected_value_section}. Sections \ref{n_alpha_section},~\ref{n_beta_section} present the sequences $n^{(1)}_i$, $n^{(2)}_i$ respectively and prove that they are as desired. \\

{\it Remark.} It is easy to see, using the union bound, that with asymptotical probability 1 in $G(n,p)$, there are no three sets $X_1,X_2,X_3$ such that each $X_i$ dominates $[n]\setminus(X_1\sqcup X_2\sqcup X_3)$ and there are no edges between distinct $X_i$ and $X_j$. It means that there exists a sequence $\left\{n_{i}\right\}_{i}$ such that, with a probability that is bounded away from 0  for large enough $i$, one can remove at most 3 edges from $G(n_i,p)$ such that the modified graph and $G(n_i,p)$ are EMSO(FO$^2$)--distinguishable. On the other hand, it is impossible to remove a bounded number of edges from $G(n,p)$ to make it FO--distinguishable from the original graph (with a probability that is bounded away from $0$  for large enough $n$). Indeed, the FO almost sure theory $\mathcal{T}$ of $G(n,p)$ is complete and its set of axioms $\mathcal{E}$ consists of so called extension axioms (see, e.g.,~\cite{Strange}). It is straightforward that all axioms from $\mathcal{E}$ hold with asymptotical probability 1 after a deletion of any bounded set of edges from $G(n,p)$. From the completeness and the FO 0-1 law, our observation follows.

\section{Expectation}\label{expected_value_section}

Let $D_n:=\{x,y,z\geq 1:\,x+y+z\leq n\}$ %\left[\frac{\ln n}{10\ln\frac{1}{1-p}},\frac{10\ln n}{\ln\frac{1}{1-p}}\right]$
 and consider integers $k,\ell,m\in D_n$. Then, clearly,
\begin{multline}\label{expected_value}
{\sf E}X(k,\ell,m)=\frac{n!}{k!\ell!m!(n-k-\ell-m)!}(k\cdot \ell\cdot m)\times(1-p)^{k\ell+\ell m+km-3} p^3\times\\
\times\prod_{v\in[n]\setminus{(X_{1} \cup X_{2}\cup X_3)}}\left[(1-(1-p)^k)(1-(1-p)^{\ell})(1-(1-p)^m)\right]\leq
\end{multline}

\begin{multline}
\frac{n^{k+\ell+m}e^{k+\ell+m}}{k^k \ell^{\ell} m^m} \exp\biggl(\ln (k\ell m)+(k\ell+km+\ell m-3)\ln(1-p)+3\ln p\\
-(n-k-\ell-m)[(1-p)^k+(1-p)^{\ell}+(1-p)^m]\biggr)= e^{f(k,\ell,m)+g(k,\ell,m)},
\label{expectation_bound_f}
\end{multline}
where $f$ and $g$ are two functions defined on $D_n$ as follows:
\begin{align}
f(k,\ell,m)= & \, k\ln(n/k)+\ell\ln(n/\ell)+m\ln(n/m)+\ln(k\ell m)+k+\ell+m\notag\\ &-n((1-p)^k+(1-p)^{\ell}+(1-p)^m)+(k\ell+km+\ell m-3)\ln(1-p)+3\ln p,
\label{expectation:power_of_exponent}
\end{align}
\begin{equation}
g(k,\ell,m)= (k+\ell+m)[(1-p)^k+(1-p)^{\ell}+(1-p)^m].
\label{expectation:power_of_exponent_reminder}
\end{equation}

%Let us prove that, for $n$ large enough, $f$ has a maximum on $D_n$ and find it.  

Let us now compute the partial derivatives:
\begin{align*}
&\frac{\partial f}{\partial k}=\ln\frac{n}{k}+(\ell+m)\ln(1-p)+\frac{1}{k}
-n(1-p)^k\ln(1-p)+1,\\
&\frac{\partial^{2} f}{\partial k^{2}}=-\frac{1}{k}-\frac{1}{k^2}-n(1-p)^k\ln^2(1-p),\\
&\frac{\partial^{2} f}{\partial k \partial \ell}=\frac{\partial^{2} 
 f}{\partial \ell \partial m}=\frac{\partial^{2} f}{\partial k \partial m}=\ln(1-p).
\end{align*}
Other derivatives can be obtained by using the symmetry of $f$. Let us find $k^*$ such that $\left.\frac{\partial f}{\partial k}\right|_{(k^*,k^*,k^*)}=\left.\frac{\partial f}{\partial \ell}\right|_{(k^*,k^*,k^*)}=\left.\frac{\partial f}{\partial m}\right|_{(k^*,k^*,k^*)}=0$. There is exactly one such $k^*$ since the equation
$$
 \ln\frac{n}{k}+2k\ln(1-p)+\frac{1}{k}
-n(1-p)^k\ln(1-p)+1=0
$$
has the unique solution
\begin{equation}
 k^*=\frac{\ln n-\ln\ln n+\ln\ln\frac{1}{1-p}}{\ln\frac{1}{1-p}}+O\left(\frac{\ln\ln n}{\ln n}\right).
\label{k_star}
\end{equation}
Let us show that $A=(k^*,k^*,k^*)$ is a point of local maximum of $f$ for all $n$ large enough. Consider the Hessian matrix 
\begin{equation*}
C=\left(
\begin{array}{cccc}
\left.\frac{\partial^{2} f}{\partial k^2}\right|_A &\left.\frac{\partial^{2} f}{\partial k \partial \ell}\right|_A& \left.\frac{\partial^{2} f}{\partial k \partial m}\right|_A \\
\\
\left.\frac{\partial^{2} f}{\partial k \partial \ell}\right|_A & \left.\frac{\partial^{2} f}{\partial \ell^2}\right|_A & \left.\frac{\partial^{2} f}{\partial\ell \partial m}\right|_A \\
\\
\left.\frac{\partial^{2} f}{\partial k \partial m }\right|_A & \left.\frac{\partial^{2} f}{\partial\ell \partial m}\right|_A & \left.\frac{\partial^{2} f}{\partial m^2}\right|_A \\
\end{array}
\right)=\ln(1-p)\left(
\begin{array}{cccc}
\ln n(1+o(1)) & 1 & 1 \\
\\
1 & \ln n(1+o(1)) & 1 \\
\\
1 & 1 & \ln n(1+o(1)) \\
\end{array}
\right).
\end{equation*}
By Sylvester's criterion \cite[Theorem 7.2.5]{Sylvest}, it is negative-definite for all $n$ large enough: the leading principal minors equal
$$
 \ln(1-p)\ln n(1+o(1))<0,
$$
$$
 \det\left[\ln(1-p)\left(\begin{array}{cc}
\ln n(1+o(1)) & 1 \\
1 & \ln n(1+o(1))
\end{array}\right)\right]
=\ln^2(1-p)\ln^2n(1+o(1))>0,
$$
$$
\det C=\ln^3(1-p)\ln^3 n(1+o(1))<0.
$$
Therefore, $A$ is indeed a local maximum point.\\

%Let us now note that there are no other local maximums inside the cube 
%$$
%\mathcal{C}_n:=\{(k,\ell,m)\in D_n:\min\{|k-k^*|,|\ell-k^*|,|m-k^*|\}<1.
%$$
%Indeed, as we noticed above, there are no local maximums (other than $A$) of the form $(k,k,k)$. Assume that there exist $k_0\leq\ell_0\leq m_0$ such that $\frac{\partial f}{\partial k}=\frac{\partial f}{\partial \ell}=\frac{\partial f}{\partial m}=0$ at $(k_0,\ell_0,m_0)\in\mathcal{C}_n$ and $k_0<m_0$. By symmetry, the same is true at $(m_0,\ell_0,k_0)$. Then $\left.\frac{\partial f}{\partial k}\right|_{(k_0,\ell_0,m_0)}=\left.\frac{\partial f}{\partial k}\right|_{(m_0,\ell_0,k_0)}$ implies that
%$$
% \ln\frac{n}{k_0}+\frac{1}{k_0}+n(1-p)^{k_0}\ln\frac{1}{1-p}+k_0\ln\frac{1}{1-p}=  \ln\frac{n}{m_0}+\frac{1}{m_0}+n(1-p)^{m_0}\ln\frac{1}{1-p}+m_0\ln\frac{1}{1-p}.
%$$
%But this is impossible since the function $-\ln x+\frac{1}{x}+n(1-p)^x\ln\frac{1}{1-p}+x\ln\frac{1}{1-p}$ is monotonically decreasing on $(k^*-1,k^*+1)$.\\

We have
\begin{align*}
f(A) &= 3k^*\left(\ln n-\ln k^*\right)-3n(1-p)^{k^*}+3(k^{*})^2\ln(1-p)+3 k^*+O(\ln\ln n)\\
& =3k^*\left(\ln n-\ln k^*+k^*\ln (1-p)+1\right)-\frac{3\ln n}{\ln\frac{1}{1-p}}+O(\ln\ln n)\\
& =3k^*-\frac{3\ln n}{\ln\frac{1}{1-p}}+O(\ln\ln n)=O(\ln\ln n).
\end{align*}

Notice that $k^*$ is not necessarily an integer. In Section~\ref{n_alpha_section}, we show that $n$ can be chosen in a way such that $k^*=\lfloor k^*\rfloor+\frac{1}{2}+o(1)$. In this case, the following lemma appears to be useful for bounding from above ${\sf E}X(k,\ell,m)$ for all possible $k,\ell,m$ (in particular, it implies that, for such $n$, $f(A)$ bounds from above $f(k,\ell,m)$ for all integer $(k,\ell,m)\in D_n$).\\

%The fact that $f(A)$ is global maximum of $f$ on $D_n$ follows immediately from\\

\begin{lemma}
Uniformly over all $(k,\ell,m)\in D_n$ such that $\min\{|k-k^*|,|\ell-k^*|,|m-k^*|\}\geq\frac{1}{2}+o(1)$, 
\begin{equation}\label{expectations_above_small}
f(k,\ell,m)\leq -\frac{\ln\frac{1}{1-p}}{2}\ln n(1+o(1))\biggl[(k-k^*)^2+(\ell-k^*)^2+(m-k^*)^2\biggr].
\end{equation}
%\begin{equation}
%\label{expectations_above_small}
%f(k,\ell,m)\leq
%\begin{cases}
%-\frac{3\ln n}{8}(1+o(1)), & \text{for } \min\{|k-k^*|,|\ell-k^*|,|m-k^*|\}\geq\frac{1}{2}+o(1)\\
%-3\ln n, & \text{for } \max\{|k-k^*|,|\ell-k^*|,|m-k^*|\}\geq 5
%\end{cases}
%\end{equation}
\label{lm:delta}
\end{lemma}

\begin{proof}
Let us set $\Delta_1=k-k^*,\Delta_2=\ell-\ell^*,\Delta_3=m-m^*$. Due to (\ref{expectation:power_of_exponent}),
%\begin{align*}
%f(k,\ell,m)= & k\ln(n/k)+\ell\ln(n/\ell)+m\ln(n/m)+\ln(k\ell m)+k+\ell+m\\ &-n((1-p)^k+(1-p)^{\ell}+(1-p)^m)-(k\ell+km+\ell m-3)\ln\frac{1}{1-p}+3\ln p.
%\end{align*}
%With this notation 
\begin{align*}
f(k,\ell,m)-f(k^*,k^*,k^*)\quad\quad\leq\quad-\ln\frac{1}{1-p}(\Delta_1\Delta_2+\Delta_1\Delta_3+\Delta_2\Delta_3)\quad&+\\
\sum_{i=1}^3\left(\Delta_{i}\ln n-\ln\frac{(k^*+\Delta_i)^{k^*+\Delta_i}}{(k^*)^{k^*}}-n(1-p)^{k^*}\left((1-p)^{\Delta_{i}}-1\right)-2\Delta_{i}k^*\ln\frac{1}{1-p}+2|\Delta_i|\right)&\leq\\
\sum_{i=1}^3\left(\frac{\Delta_i^2\ln\frac{1}{1-p}}{2}-k^*\ln\frac{k^*+\Delta_i}{k^*}-\Delta_i\ln k^*-\frac{\ln n\left((1-p)^{\Delta_{i}}-1+o(1)\right)}{\ln\frac{1}{1-p}}-\Delta_{i}\ln n(1+o(1))\right)&,
\end{align*}
where the last inequality follows from the inequalities $-\Delta_1\Delta_2-\Delta_1\Delta_3-\Delta_2\Delta_3\leq\frac{1}{2}(\Delta_1^2+\Delta_2^2+\Delta_3^2)$ and $-\Delta_i\ln\left(1+\frac{\Delta_i}{k^*}\right)\leq 0$.

Notice that $-k^*\ln(1+\Delta_i/k^*)\leq-\Delta_i\ln k^* I(\Delta_i\leq 0)$. Indeed, for positive $\Delta_i$, the inequality is obvious. If $\Delta_i\leq 0$, then it is sufficient to verify the inequality only for boundary values $\Delta_i=0$ and $\Delta_i=1-k^*$ (the function $-k^*\ln(1+x/k^*)+x\ln k^*$ changes its monotonicity only once on $[1-k^*,0]$: first, it decreases and, after $x=k^*/\ln k^*-k^*$, it increases). We get
$$
f(k,\ell,m)-f(k^*,k^*,k^*)\leq\sum_{i=1}^3\left(\frac{\Delta_i^2\ln\frac{1}{1-p}}{2}-\Delta_{i}\ln n-\frac{\ln n}{\ln\frac{1}{1-p}}\left((1-p)^{\Delta_{i}}-1\right)\right)(1+o(1))=
$$
$$
\left[\sum_{i=1}^3\frac{\Delta_i^2\ln\frac{1}{1-p}}{2}
-\frac{\ln n}{\ln\frac{1}{1-p}}\sum_{i=1}^3 \gamma\left(\Delta_i\ln\frac{1}{1-p}\right)\right](1+o(1))\leq
\frac{1+o(1)}{2}\ln\frac{1}{1-p}\sum_{i=1}^3\Delta_i^2(1-\ln n),
$$
where $\gamma(x)=x+e^{-x}-1\leq x^2/2$ for all $x>0$. Inequality~(\ref{expectations_above_small}) follows.

%$f(k,\ell,m)-f(k^*,k^*,k^*)$ does not exceed
%\begin{gather*}
%\sum_i\left(\Delta_{i}\ln n-\ln\frac{(k^*+\Delta_i)^{k^*+\Delta_i}}{(k^*)^{k^*}}-n(1-p)^{k^*}\left((1-p)^{\Delta_{i}}-1\right)-2\Delta_{i}k^*\ln\frac{1}{1-p}+2|\Delta_i|\right)\leq\\
%\sum_i\left(\Delta_{i}\ln n-k^*\ln\frac{(k^*+\Delta_i)}{k^*}-\frac{\ln n}{\ln\frac{1}{1-p}}\left((1-p)^{\Delta_{i}}-1\right)-2\Delta_{i}\ln n-10|\Delta_i|\right)\leq
%\end{gather*}
%\begin{multline*}
%\sum_i\left(-\Delta_{i}\ln n-\frac{\ln n}{\ln\frac{1}{1-p}}\left((1-p)^{\Delta_{i}}-1\right)\right)=
%-\frac{\ln n}{\ln\frac{1}{1-p}}\left(\sum_i g\left(\Delta_i\ln\frac{1}{1-p}\right)\right)\leq\\-\frac{\ln n}{\ln\frac{1}{1-p}}\left(\frac{\sum_i\left(\Delta_i\ln\frac{1}{1-p}\right)^2}{2}\right)\leq-\frac{3\ln\frac{1}{1-p}\ln n}{8},
%\end{multline*}
%where we used that $g(x)=x+e^{-x}-1\leq x^2/2.$
\end{proof}

%Notice that $k^*$ is not necessarily an integer. In Section~\ref{n_alpha_section}, we show that $n$ can be chosen in a way such that $k^*=\lfloor k^*\rfloor+\frac{1}{2}+o(1)$. In this case, Lemma~\ref{lm:delta} appears to be useful for bounding from above ${\sf E}X(k,\ell,m)$ for all possible $k,\ell,m$.\\

\section{A sequence of small probabilities}\label{n_alpha_section}

Let us find a sequence $(n^{(1)}_i,\,i\in\mathbb{N})$ such that ${\sf P}(\exists k,\ell,m\,\,X(k,\ell,m)>0) \rightarrow 0$ as $i \rightarrow \infty$. 

For $i \in \mathbb{N},$ set 
$$
n:=n^{(1)}_i=\left\lfloor\left(\frac{1}{1-p}\right)^{i+\frac{1}{2}}i\right\rfloor.
$$ 

Clearly, $k^{*}=k^{*}(n)=i+\frac{1}{2}+o(1)$ ($k^*$ is defined in (\ref{k_star})).

%Let $k,\ell,m \in \mathbb{N}$. Then the requirements of Lemma~\ref{expectations_above_small} are satisfied. 
Using Lemma~\ref{lm:delta}  and inequality (\ref{expectation_bound_f}), we get that, uniformly over all $k,\ell,m\in D_n$, 
$$
{\sf E}X(k,\ell,m)\leq e^{-\frac{1}{2}\ln\frac{1}{1-p}\ln n(1+o(1))\biggl[(k-k^*)^2+(\ell-k^*)^2+(m-k^*)^2\biggr]+g(k,\ell,m)}.
$$
Notice that 
$$
g(k,\ell,m)<3\biggl[|k-k^*|+|\ell-k^*|+|m-k^*|\biggr]+3k^*\biggl[(1-p)^k+(1-p)^{\ell}+(1-p)^m\biggr]$$ 
and 
$$
3k^*\biggl[(1-p)^k+(1-p)^{\ell}+(1-p)^m\biggr]=o(1)\ln n\biggl[(k-k^*)^2+(\ell-k^*)^2+(m-k^*)^2\biggr].
$$
Therefore, 
$$
{\sf E}X(k,\ell,m)\leq e^{-\frac{1}{2}\ln\frac{1}{1-p}\ln n(1+o(1))\biggl[(k-k^*)^2+(\ell-k^*)^2+(m-k^*)^2\biggr]}.
$$
By the union bound and Markov's inequality,
$$
 {\sf P}\biggl(\exists k,\ell,m\in D_n\quad X(k,\ell,m)>0\biggr)\leq\sum_{k,\ell,m\in D_n}{\sf E}X(k,\ell,m)\leq
 \left[\sum_{j=1}^{\infty}e^{-\frac{j}{8}\ln\frac{1}{1-p}\ln n(1+o(1))}\right]^3=o(1).
$$
%where the last sum tends to infinity as $n$ grows.
%Notice that whp all independent sets in $G(n,p)$ have size smaller than $2\frac{\ln n}{\ln\frac{1}{1-p}}$ (see, e.g.,~\cite{Independence}). Moreover, by the union bound, the probability that there exist disjoint sets $X,Y\subset[n]$ such that $|X|\leq\frac{\ln n}{2\ln\frac{1}{1-p}}$, $|Y|\leq\frac{4\ln n}{\ln\frac{1}{1-p}}$ and every vertex from $n\setminus(X\sqcup Y)$ has a neighbor in $X$ is at most
%\begin{align*}
% \sum_{x=1}^{\lfloor\frac{1}{2}\log _{\frac{1}{1-p}}n\rfloor}\sum_{y=1}^{\lfloor 4\log_{\frac{1}{1-p}}n\rfloor}n^{x+y}(1-(1-p)^x)^{n-x-y} & \leq %e^{\frac{5}{\ln\frac{1}{1-p}}\ln^2 n -n(1-p)^{[\log_{1/(1-p)}n]/2}(1+o(1))}\\
 %&=e^{\frac{5}{\ln\frac{1}{1-p}}\ln^2 n -\sqrt{n}(1+o(1))}\to 0.
% \exp\left\{\frac{5}{\ln\frac{1}{1-p}}\ln^2 n -n(1-p)^{\frac{[\log_{1/(1-p)}n]}{2}}(1+o(1))\right\}\\
% &=\exp\left\{\frac{5}{\ln\frac{1}{1-p}}\ln^2 n -\sqrt{n}(1+o(1))\right\}\to 0.
%\end{align*}
%From this, it immediately follows that
%$$
% {\sf P}\biggl(\exists (k,\ell,m)\notin (D_n)^3\quad X(k,\ell,m)>0\biggr)\to 0.
%$$
Therefore, $(n^{(1)}_i,\,i\in\mathbb{N})$ is the desired sequence.

\section{A sequence of large probabilities}\label{n_beta_section}

Here, we introduce a sequence $(n^{(2)}_i,\,i\in\mathbb{N})$, such that, for some $k=k(n^{(2)}_i)$, ${\sf P}(X(k,k,k)>0)$ is bounded away from $0$ for all $i$  large enough.
For $i\in\mathbb{N}$, define 
\begin{equation}\label{choice_n_2}
n^{(2)}_i=\left\lfloor\left(\frac{1}{1-p}\right)^i i\right\rfloor.
\end{equation}

Notice that $k^{*}=k^{*}(n^{(2)}_i)=i+o(1)$, where $k^*$ is defined in (\ref{k_star}). Setting $n=n_{k}^{(2)}$ for any $k \in \mathbb{N}$, we have
 %using (\ref{expected_value}):
\begin{equation}
\begin{split}
{\sf E}X(k,k,k) & =\frac{n!}{k!k!k!(n-3k)!}k^3(1-p)^{3k^2-3}\cdot p^3
\cdot\left[(1-(1-p)^k)\right]^{3(n-3k)}\\
&=\frac{n^{n}\sqrt{2\pi n}}{k^{3k}\sqrt{(2\pi k)^3}\cdot (n-3k)^{n-3k}\sqrt{2\pi(n-3k)}}\cdot k^3(1-p)^{3k^2-3}p^3\cdot e^{-3n(1-p)^k}(1+o(1)) \\
&=\frac{n^{3k}e^{3k}}{k^{3k}\sqrt{(2\pi)^3}}\left(\frac{p}{1-p}\right)^3 k^{3/2}(1-p)^{3k^2}e^{-3k}(1+o(1))\\
&=\left(\frac{p}{1-p}\right)^3\frac{k^{3/2}}{\sqrt{(2\pi)^3}}(1+o(1)).
\end{split}
\label{expectation_infinite_asymp}
\end{equation}
So, ${\sf E}X(k,k,k)\to\infty$ as $k \rightarrow \infty . $ It remains to prove that $[{\sf E}X(k,k,k)]^2/{\sf E}X^2(k,k,k)$ is bounded away from 0 and apply the Paley--Zygmund inequality~\cite{PZ} (stated below).

\begin{theorem}[Paley--Zygmund inequality]
Let $X$ be a non-negative random variable with ${\sf E}X^2<\infty$. Then for any $0 \leq \lambda<1$,
$$
{\sf P}[X>\lambda {\sf E} X] \geq(1-\lambda)^{2} \frac{({\sf E} X)^{2}}{{\sf E}\left[X^{2}\right]}.
$$
\label{th:PZ}
\end{theorem}

\subsection{Second moment}

%\subsection{Notation}\label{exp_value2}

Let us call a tuple $\left(X_{1},x_1,X_{2},x_2,X_3,x_3\right)$ a {\it $k$-tuple} if sets $X_1,X_2,X_3\subset[n]$ are disjoint, $|X_1|=|X_2|=|X_3|=k$ and $x_1\in X_1$, $x_2\in X_2$, $x_3\in X_3$. Let us call a $k$-tuple $\left(X_{1},x_1,X_{2},x_2,X_3,x_3\right)$ {\it special}, if it
satisfies the conditions given in Section~\ref{sec:intro}:
\begin{itemize}
\item every vertex $v$ from $[n]\setminus{(X_{1} \cup X_{2}\cup X_3)}$ has at least one neighbor in each $X_i$,
\item for any distinct $i,j\in\{1,2,3\}$, there is exactly one edge between $X_i$ and $X_j$ --- namely, the edge between $x_i$ and $x_j$.
\end{itemize}

Let 
\begin{equation}
\mathcal{X}=(X_1,x_1,X_2,x_2,X_3,x_3) \qquad \text{and} \qquad \mathcal{Y}=(Y_1,y_1,Y_2,y_2,Y_3,y_3),
\label{two_tuples}
\end{equation}
be two $k$-tuples. Everywhere below, we denote 
$$
r:=|(X_1\sqcup X_2\sqcup X_3)\cap (Y_1\sqcup Y_2\sqcup Y_3)|,
$$
$$
r_i:=|Y_i\cap (X_1\sqcup X_2\sqcup X_3)|,\,\,\,
r_{j+3}:=|X_j\cap (Y_1\sqcup Y_2\sqcup Y_3)|,\,\,\,
r_{i,j}:=|Y_i\cap X_j|.
$$  

Let $\Gamma$ be the set of all $k$-tuples. For $\mathcal{X}\in\Gamma$, let $\xi_{\mathcal{X}}$ be the Bernoulli random variable that equals 1 if and only if $\mathcal{X}$ is special. Then $X(k,k,k)=\sum_{\mathcal{X}\in\Gamma}\xi_{\mathcal{X}}$. From this,
$$
 {\sf E}X^2(k,k,k)=\sum_{\mathcal{X},\mathcal{Y}\in\Gamma}\xi_{\mathcal{X}}\xi_{\mathcal{Y}}.
$$

%\subsection{Second moment}

We compute this value in the usual way by dividing the summation into parts with respect to the value of $r$:
\begin{equation}
{\sf E}X^2(k,k,k)=S_1+S_2+S_3,
\label{second_moment_S}
\end{equation}
\begin{itemize}
	\item $S_1=\sum_{\mathcal{X},\mathcal{Y}\in\Gamma:\,\,r\in(r_0,3k-r_0)}\xi_{\mathcal{X}}\xi_{\mathcal{Y}}$,
	\item $S_2=\sum_{\mathcal{X},\mathcal{Y}\in\Gamma:\,\,r\leq r_0}\xi_{\mathcal{X}}\xi_{\mathcal{Y}}$,
	\item $S_3=\sum_{\mathcal{X},\mathcal{Y}\in\Gamma:\,\,r\geq 3k-r_0}\xi_{\mathcal{X}}\xi_{\mathcal{Y}}$,
\end{itemize}
where $r_0=\left\lceil \frac{16}{\ln[1/(1-p)]}\right\rceil$.\\

In Section~\ref{sec:intersection_small}, we give upper bounds on $S_1$ and $S_2$. An upper bound on $S_3$ is given in Section~\ref{sec:intersection_large}. In Section~\ref{sec:P-Z}, we apply the Paley--Zygmund inequality and finish the proof. Auxiliary lemmas that are used for bounds on $S_i$ are given in Section~\ref{sec:aux}.
 
\subsection{Auxiliary lemmas}
\label{sec:aux}

%Let $r\in\{0,1,\ldots,3k\}$.
For a $k$-tuple $\mathcal{X}=(X_1,x_1,X_2,x_2,X_3,x_3)$, let $\mathcal{N}_{\mathcal{X}}$ be the event saying that there are no edges between $X_1,X_2,X_3$ except for those between $x_1,x_2,x_3$.

\begin{lemma}\label{no_edges_bound}
    Let $\mathcal{X},\mathcal{Y}$ be $k$-tuples. Then 
    $$
    {\sf P}(\mathcal{N}_{\mathcal{Y}}|\mathcal{N}_{\mathcal{X}})\leq(1-p)^{3k^2-\frac{r^2}{3}-3}.
    $$
\end{lemma}

\begin{figure}[h]
	\center{\includegraphics[scale=0.3]{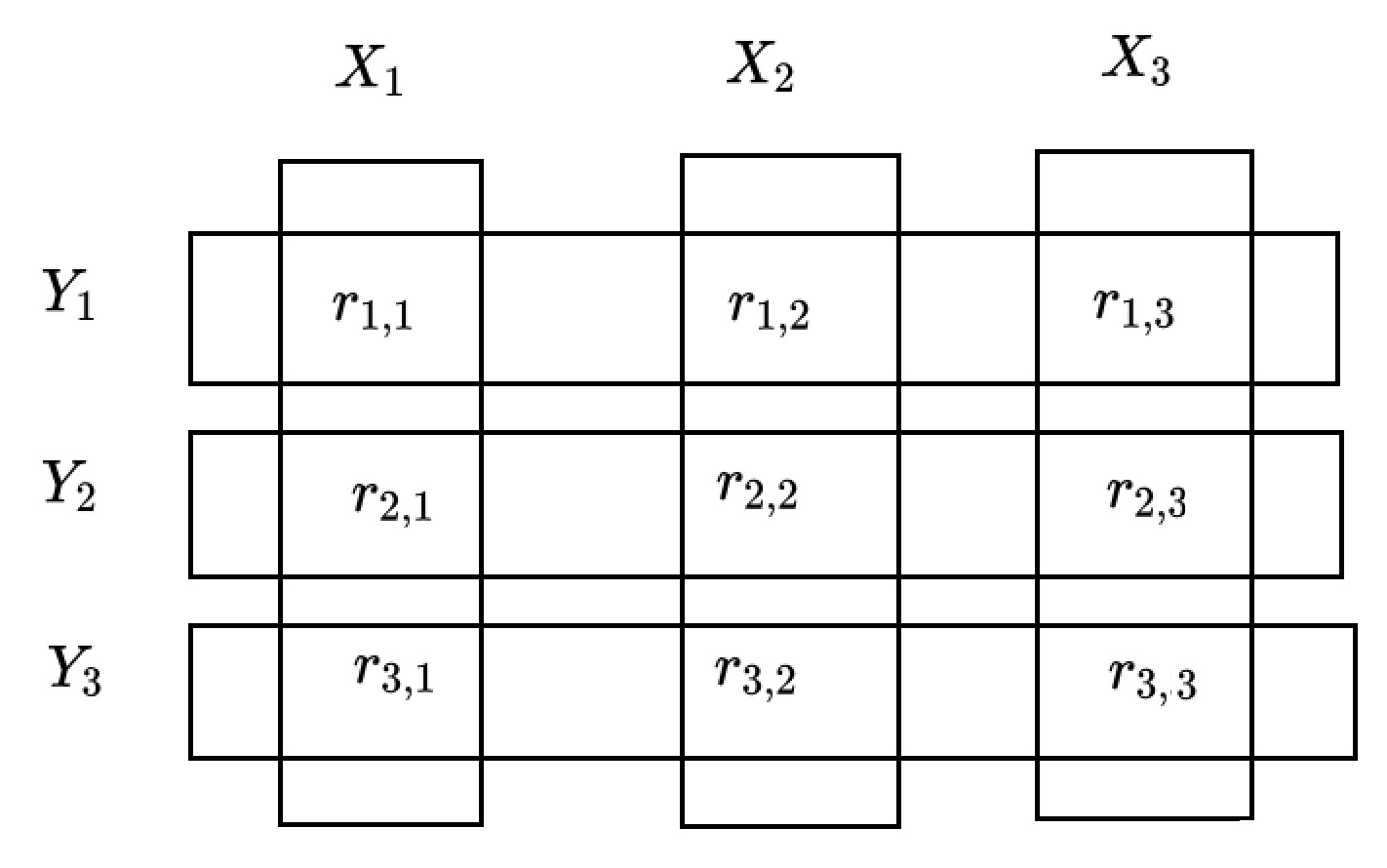}}
	\label{fig:image}
	\caption{two $k$-tuples $X$ and $Y$ with given intersections $r_{i,j}$,  $i,j\in\{1,2,3\}$.}
\end{figure}

\begin{proof}
	The number of pairs $(u,v)\in[n^2]$ such that $u\in Y_i, v\in Y_j$ for some  $i\neq j$ but $u$ and $v$ do not belong to $X_{\tilde i},X_{\tilde j}$ for any distinct $\tilde i, \tilde j\in\{1,2,3\}$ equals
%	\begin{gather*}
%	3k^2-r_{1,1}(r_{2,2}+r_{2,3}+r_{3,2}+r_{3,3})-r_{1,2}(r_{2,1}+r_{2,3}+r_{3,1}+r_{3,3})-r_{1,3}(r_{2,1}+r_{2,2}+r_{3,1}+r_{3,2})-\\
%	r_{2,1}(r_{3,2}+r_{3,3})-r_{2,2}(r_{3,1}+r_{3,3})-r_{2,3}(r_{3,1}+r_{3,2})=\\ 
    $$  
    3k^2-\frac{1}{2}\sum_{i,j} r_{i,j}\sum_{\tilde i\neq i,\tilde j\neq j} r_{\tilde i,\tilde j}\geq 3k^2-(r_1r_2+r_2r_3+r_3r_1)\geq 3k^2-\frac{r^2}{3},
    $$	
    where we used that
%3k^2-\frac{1}{4}(Ax,x),
%	\end{gather*}
%	Applying inequalities given below (inequality \ref{FKG_bound}), we bound the number of such pairs by $3k^2-r^2/3$. 
	\begin{equation}\label{FKG_bound}
	\frac{r^2}{9}=\left(\frac{r_1+r_2+r_3}{3}\right)^2\geq
	%\text{|uncover brackets and apply FKG|}\geq
	\frac{r_1r_2+r_2r_3+r_3r_1}{3}.
	\end{equation}
Finally, it remains to notice that conditional probability ${\sf P}(\mathcal{N}_{\mathcal{Y}}|\mathcal{N}_{\mathcal{X}})$ does not exceed $(1-p)^{3k^2-\frac{r^2}{3}-3}$ since we should exclude no more than 3 pairs of vertices $(u,v)$ that coincide with a pair of vertices from $y_1,y_2,y_3$.
\end{proof}

\begin{lemma}\label{comment}
Let $\mathcal{X},\mathcal{Y}$ be $k$-tuples and there exists $s\in\{1,2,3\}$ such that $r_s>k-r_0$ and $r_{s,j}<k-6r_0$ for all $j\in\{1,2,3\}$. If $r_0<\frac{1}{30}k$, then
$$
{\sf P}(\mathcal{N}_{\mathcal{Y}}|\mathcal{N}_{\mathcal{X}})\leq(1-p)^{3k^2-\frac{r^2}{3}-3+4kr_0}.
$$
\end{lemma}
\begin{proof}
Repeating previous arguments, it is sufficient to prove that $\frac{1}{2}\sum_{i,j} r_{i,j}\sum_{\tilde i\neq i,\tilde j\neq j} r_{\tilde i,\tilde j}\leq\frac{r^2}{3}-4kr_0$. Without loss of generality, let us assume that $r_{s,1}\geq r_{s,2}\geq r_{s,3}$.

Applying (\ref{FKG_bound}) for $r_4,r_5,r_6$, we get
	\begin{multline}
	\frac{1}{2}\sum_{i,j} r_{i,j}\sum_{p\neq i,q\neq j} r_{p,q}=
	(r_4r_5+r_4r_6+r_5r_6)-\sum_{i=1}^3(r_{i,1}r_{i,2}+r_{i,2}r_{i,3}+r_{i,1}r_{i,3})\leq\\ (r_4r_5+r_4r_6+r_5r_6)-(r_{s,1}r_{s,2}+r_{s,1}r_{s,3}+r_{s,2}r_{s,3})\leq\\
	\frac{r^2}{3}-(r_{s,1}r_{s,2}+r_{s,1}r_{s,3}+r_{s,2}r_{s,3})\leq\frac{r^2}{3}-r_{s,1}(r_s-r_{s,1})<\frac{r^2}{3}-r_{s,1}(k-r_0-r_{s,1})\leq\\
	\frac{r^2}{3}-(k-6r_0)(k-r_0-(k-6r_0))=
	\frac{r^2}{3}-5r_0(k-6r_0)\leq\frac{r^2}{3}-4r_0k
	\end{multline}
since the function $f(x)=x(k-r_0-x)$ is concave on $[\frac{k-r_0}{3},k-6r_0]$, and so, it  achieves the minimum value at one of the ends of the segment. Clearly, the value at the right end is smaller.
%
%In the end, subtract $r^2/3-4r_0k+o(1)$ from $3k^2$.
\end{proof}

\begin{lemma}\label{dominatong set bound}
	Let $r_0>0$ be a fixed number. Let $\mathcal{X},\mathcal{Y}$ be $k$-tuples~(\ref{two_tuples}) and $r\leq r_0$. Then the probability that each $X_j$, $j\in\{1,2,3\}$, and each $Y_i$, $i\in\{1,2,3\}$, are dominating sets in $[n]\setminus((X_1\sqcup X_2\sqcup X_3)\cup(Y_1\sqcup Y_2\sqcup Y_3))$ does not exceed $(1-6(1-p)^k)^{n}(1+o(1))$.
\end{lemma}
\begin{proof}
%	Let $r_{j+3}$ denotes the cardinality $|X_j\cap Y|.$ 
Fix a vertex $v\in[n]\setminus((X_1\sqcup X_2\sqcup X_3)\cup(Y_1\sqcup Y_2\sqcup Y_3))$. Set $X:=X_1\sqcup X_2\sqcup X_3$, $Y:=Y_1\sqcup Y_2\sqcup Y_3$. Then $\{v$ has neighbors in each $X_j$ and each $Y_i\}\subset\mathcal{A}\cup\mathcal{B}\cup\mathcal{C}$, where
	\begin{itemize}
		\item $\mathcal{A}=\biggl\{\forall i\in\{1,2,3\}\quad v$ has neighbors both in  $X_i\setminus Y$ and in $Y_i\setminus X\biggr\}$,
		\item $\mathcal{B}=\biggl\{\exists i\in\{1,2,3\}\quad v$ has a neighbor in $X_i\cap Y$ and does not have a neighbor in $X_i\setminus  Y\biggr\}$,
		\item $\mathcal{C}=\biggl\{\exists i\in\{1,2,3\}\quad v$ has a neighbor in $Y_i\cap X$ and does not have a neighbor in $Y_i\setminus X\biggr\}$.
	\end{itemize}
	Clearly, 
	$$
	{\sf P}(\mathcal{A})=\prod_{i=1}^6(1-(1-p)^{k-r_i}).
	$$ 
	$$
	{\sf P}(v\text{ has a neighbor in }Y_i\cap X\text{ and does not have a neighbor in }Y_i\setminus X)=(1-p)^{k-r_i}(1-(1-p)^{r_i}),
	$$
	$$
	{\sf P}(v\text{ has a neighbor in }X_j\cap Y\text{ and does not have a neighbor in }X_j\setminus Y)=(1-p)^{k-r_{j+3}}(1-(1-p)^{r_{j+3}}).
	$$ 
	Therefore,	
	\begin{gather*}
	{\sf P}(\mathcal{A}\cup\mathcal{B}\cup\mathcal{C})\leq\prod_{i=1}^6(1-(1-p)^{k-r_i})+\sum_{i=1}^6 (1-p)^{k-r_i}(1-(1-p)^{r_i})=\\
	1-\sum_{i=1}^6 (1-p)^{k-r_i}+O\left((1-p)^{2k}\right)+\sum_{i=1}^6 ((1-p)^{k-r_i}-(1-p)^k)=1-6(1-p)^k+O\left((1-p)^{2k}\right).
	\end{gather*}
	So, multiplying over $v\in n\setminus((X_1\sqcup X_2\sqcup X_3)\cup(Y_1\sqcup Y_2\sqcup Y_3))$ and recalling that  $n=\lfloor k/(1-p)^k\rfloor$, we get the desired bound.
\end{proof}

\subsection{Small and medium intersections}
\label{sec:intersection_small}

In this section, we estimate $S_1$ and $S_2$.\\

By Lemma~\ref{no_edges_bound} and Lemma~\ref{dominatong set bound},
\begin{equation}\label{S1}
S_1\leq \sum_{r=r_0}^{3k-r_0}\binom{n}{3k}3^{3k}k^3 \binom{n}{3k-r}\binom{3k}{r} 3^{3k}k^3(1-p)^{3k^2-3}(1-p)^{3k^2-\frac{r^2}{3}-3},
\end{equation}

\begin{equation}\label{S2}
S_2\leq\sum_{r\leq r_0}\frac{n!}{k!k!k!(n-3k)!}(1-p)^{3k^2-3}k^3\sum_{s_1,s_2,s_3} \frac{n^{s_1}n^{s_2}n^{s_3}}{s_1!s_2!s_3!}(1-p)^{3k^2-\frac{r^2}{3}-3} k^3(1-6(1-p)^k)^{n}(1+o(1)),
\end{equation}

where the second summation is over all non-negative integers $s_1,s_2,s_3$ such that $s_1+s_2+s_3=3k-r$ and, for each $i\in\{1,2,3\}$, $|s_i-k|\leq r_0$.\\

%$Q$ is an upper bound on the probability that each $X_i$ and each $Y_j$ is dominating sets in $[n]\setminus(X\cup Y)$.

%\bigskip
%In (\ref{S1})-(\ref{S2}) factors:
%\begin{itemize}
%	\item $\binom{n}{3k}3^{3k}k^3\binom{n}{3k-r}\binom{3k}{r} 3^{3k}k^3$ is the number of pairs of $6$-tuples sharing $r$ elements.
%	\item $(1-p)^{3k^2}$ is probability that there is no edges between $X_1,X_2$ and $X_3.$
%	\item $(1-p)^{3k^2-\frac{r^2}{3}}$ is maximum value of conditional probability that there is no edges between $Y_1,Y_2$ and $Y_3$, given there is no edges between $X_1,X_2$ and $X_3.$ (see Lemma \ref{no_edges_bound}).
%	\item $\sum_{r\leq r_0} \sum \frac{n^{s_1}n^{s_2}n^{s_3}}{s_1!s_2!s_3!}\frac{n!}{k!k!k!(n-3k)!}k^6$ is the number of pairs of $6$-tuples sharing at most $r_0$ elements.
%\end{itemize}

%\subsubsection{Calculating $S_1$ and $S_2$}

%\bigskip

From (\ref{S1}), we get that, for $k$ large enough,
\begin{gather*}
S_1\leq
\sum_{r=r_0}^{ 3k-r_0}\left(\frac{en}{3k}\right)^{3k}\left(\frac{en}{3k-r}\right)^{3k-r}2^{3k} 3^{6k}k^{6}(1-p)^{6k^2-6-\frac{r^2}{3}}\leq\\
\sum_{r=r_0}^{ 3k-r_0}\left(\frac{n}{k}\right)^{6k-r}\frac{\left(3^6 2^3\right)^k k^6 e^{6k-r}}{\left(1-\frac{r}{3k}\right)^{3k-r}} (1-p)^{6k^2-6-\frac{1}{3}r^2}
%\text{|put $n$ from \ref{choice_n_2} and  apply $1-1/(|\alpha|+1)\geq\exp\{-1/|\alpha|\}$ for estimating $(1-r/3k)^{3k-r}$|}\\
\leq\sum_{r=r_0}^{ 3k-r_0}(1-p)^{rk-6-\frac{r^2}{3}}e^{15k},
\end{gather*}
where the last inequality is obtained from 

\begin{itemize}
\item the definition of $n$ (we get $n\leq k(1-p)^{-k}$), 
\item the inequality $\ln(1-x)>-x-x^2$ that is true for all $x\in(0,1)$ (it is applied here with $x=\frac{r}{3k}$),
\item and the inequality $k^6<C^k$ that is true for any $C>1$ and large enough $k$ (it is applied here with $C=\frac{e^9}{3^62^3}>1$).
\end{itemize}

Finally, we get that
\begin{equation}
S_1<
\sum_{r=r_0}^{3k-r_0}\left((1-p)^{r-\frac{6}{k}-\frac{r^2}{3k}}\cdot e^{15}\right)^k=o(1)
\label{S1_bound}
\end{equation}
since, for $r\in[r_0,3k-r_0]$, we have $r-\frac{6}{k}-\frac{r^2}{3k}\geq r_0(1-\frac{6}{kr_0}-\frac{r_0}{3k})=r_0(1-o(1))$ and due to the choice of $r_0>\frac{16}{\ln[1/(1-p)]}$.\\

It remains to bound $S_2$. For $k$ large enough, we get
\begin{gather*}
S_2\leq
\frac{(n/k)^{3k}}{k^{3/2}\left(1-\frac{3k}{n}\right)^{n-3k}} (1-p)^{6k^2-\frac{r_0^2}{3}-6} k^6(1-6(1-p)^k)^{n}\times \sum_{r\leq r_0}\sum_{s_1,s_2,s_3} \frac{n^{s_1}n^{s_2}n^{s_3}}{s_1!s_2!s_3!} \leq\\
c_1 e^{3k}k^{9/2}(1-p)^{3k^2}(1-6(1-p)^k)^{n}
\sum_{r\leq r_0}\sum_{s_1,s_2,s_3}\frac{n^{3k-r}}{s_1!s_2!s_3!}
\end{gather*}
for some positive constant $c_1$, where the last inequality follows from the definition of $n$ and the inequality $\ln(1-x)\geq-\frac{x}{1-x}$ applied to $x=\frac{3k}{n}$.

Notice that, for $s_1,s_2,s_3$ such that $s_1+s_2+s_3=3k-r$ and $|s_i- k|\leq r_0$, we get that 
$$
s_1!s_2!s_3!\geq\sqrt{s_1s_2s_3}s_1^{s_1}s_2^{s_2}s_3^{s_3}e^{-s_1-s_2-s_3}\geq\sqrt{(k-r_0)^3}\left(k-\frac{r}{3}\right)^{3k-r}e^{r-3k}
$$
since the minimum value of $s_1^{s_1}s_2^{s_2}s_3^{s_3}$ is achieved when $s_1=s_2=s_3$. Therefore, we get
\begin{align}
S_2\,\,&\leq\, c_2e^{3k}k^3(1-p)^{3k^2}\left(1-6(1-p)^k\right)^{n}
\sum_{r\leq r_0}\left(\frac{en}{k-r/3}\right)^{3k-r}\notag\\
&\leq\, r_0c_2e^{3k}k^{3}(1-p)^{3k^2}e^{-6n(1-p)^k}\left(\frac{en}{k}\right)^{3k}
\leq\,\, c_3k^3
\label{S2_bound}
\end{align}
for some positive constants $c_2$ and $c_3$.

%Next section we will deal with the last case, when the intersection $X$ and $Y$ is close to $3k=|Y|=|X|$. 

\subsection{Large intersections}
\label{sec:intersection_large}

In this section we produce an upper bound for $S_3$. We let $S_3=S_3^1+S_3^2$, where $S_3^1$ is the summation over all $\mathcal{X},\mathcal{Y}$ such that $r>3k-r_0$ and for each $Y_i$ there exists $X_j$ which almost coincides with $Y_i$ (see Figure~\ref{fig:two_types_of_large_intersections}):
\begin{equation}
\forall i\in\{1,2,3\}\,\,\exists j\in\{1,2,3\}:\quad  r_{i,j}\geq k-6r_0.    
\label{condition_one}
\end{equation}

%We split the calculation of $S_3$ into two different cases. The idea of such splitting is that either for each $X_i,i\in\{1,2,3\}$ there exists $Y_j,j\in\{1,2,3\}$, which almost coincide with $X_i$ or there exists $X_i, i\in\{1,2,3\}$ which is rather different from each $Y_j, j\in\{1,2,3\}$. Now, in more detail:

\begin{figure}[h]
	\center{\includegraphics[scale=0.3]{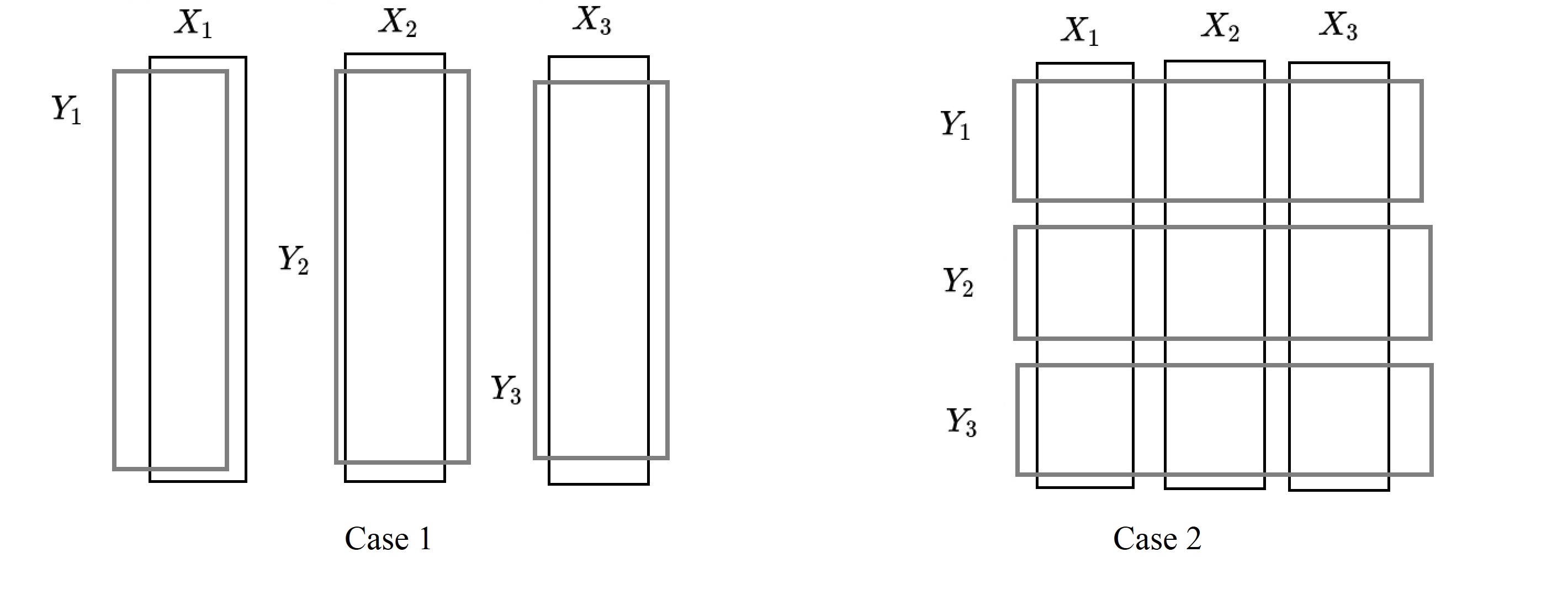}}
	\caption{two types of intersections between $Y_i$ and $X_j$. Case 1 is presented in simplified form (in general, each $Y_i$ may have a non-empty intersection with each $X_j$).}
	\label{fig:two_types_of_large_intersections}
\end{figure}

%\begin{caseof}
%	\case{Let $r>3k-r_0$ and for all $i\in\{1,2,3\}$ there exists $j\in\{1,2,3\},$ such as $r_{i,j}>(k-6r_0)$.} Without loss of generality, let us assume that $r_{1,1}\geq r_{2,2}\geq r_{3,3}\geq k-6r_0$. 

Notice that given $\mathcal{X}$ and $r_4,r_5,r_6$, the number of ways to choose $Y_i\cap(X_1\sqcup X_2\sqcup X_3)$, $i\in\{1,2,3\}$, is bounded from above by 
$$
\binom{k}{k-r_4}\binom{k}{k-r_5}\binom{k}{k-r_6} 3^{3k-r}\leq
(3k)^{3k-r}.
%\frac{k^{r_1+r_2+r_3}}{r_1!r_2!r_3!}3^{r_0}\leq
%\frac{(ke)^r}{r_1^{r_1}r_2^{r_2}r_3^{r_3}}3^{r_0}\leq
%\frac{(ke)^r}{(k-r_0)^r}3^{r_0}=O(e^r)
$$
%since $r_i\geq k-r_0$ for all $i\in\{1,2,3\}$.

Also, given $\mathcal{X}$ and $Y_1,Y_2,Y_3$ such that (\ref{condition_one}) holds, the number of choices of $y_1,y_2,y_3$ such that $(Y_1,y_1,Y_2,y_2,Y_3,y_3)$ has a chance to be special is $O(k)$ since, for every possible choice of $(y_1,y_2,y_3)$, there exists $j\in\{1,2,3\}$ such that, for every $i\in\{1,2,3\}$, $y_i$ belongs to either $\{x_1,x_2,x_3\}$, or $(Y_1\sqcup Y_2\sqcup Y_3)\setminus (X_1\sqcup X_2\sqcup X_3)$ (this set has a bounded size), or $X_j$. Indeed, it is not possible that two $y$-vertices belong to different $X$-sets and do not belong to $\{x_1,x_2,x_3\}$ because there are no edges between $X_1,X_2,X_3$ other than those between $x_1,x_2,x_3$.

Finally, given $\mathcal{X}$ and $\mathcal{Y}$, 
$$
 {\sf P}(\mathcal{N}_{\mathcal{Y}}|\mathcal{N}_{\mathcal{X}})\leq (1-p)^{-3+\sum_{i\in\{1,2,3\}}(k-r_i)(r-r_i)}.
$$
Since $r_1+r_2+r_3=r$, we get
$$
	\sum_i (k-r_{i})(r-r_{i})=3kr-k\sum_i r_{i}-r^2+\sum_i r_{i}^2=2kr-r^2+\sum_i r_{i} \geq 2kr-r^2+r^2/3=2r(3k-r)/3.
$$

Combining all the above bounds, we get that there exists $C>0$ such that
	\begin{multline}\label{S_3_1}
	S_3^1\leq \frac{n!}{k!k!k!(n-3k)!} p^3(1-p)^{3k^2-3} k^3 (1-(1-p)^k)^{3(n-6k)}\times\\ 
	Ckn^{3k-r}(3k)^{3k-r} (1-p)^{2r(3k-r)/3}.
	\end{multline}
	
%	In (\ref{S_3_1}) factor:
%where
%  \begin{itemize}
%	\item  $C_{y_1,y_2,y_3}(r)$ is the number of choices of $\{y_1,y_2,y_3\}$.
%	Note that when $r<3k$, $C_{y_1,y_2,y_3}(r)=O(k)$ since there is no two different vertices $y_i$, $y_j, i,j\in\{1,2,3\}$ which both belong to $(X_1\sqcup X_2\sqcup X_3)\cap (Y_1\sqcup Y_2\sqcup Y_3)$ and both are different from $x_1,x_2$ and $x_3$. Moreover, when $r=3k$, $\{y_1,y_2,y_3\}$ and $\{x_1,x_2,x_3\}$ are equal up to a permutation, so $C_{y_1,y_2,y_3}(3k)\leq 6.$%
%	\item 
%$\sum_i (k-r_{i})(r-r_{i})$ is the number of all possible edges between $Y_i\setminus (X_1\sqcup X_2\sqcup X_3)$ and $(Y\setminus Y_i)\cap (X_1\sqcup X_2\sqcup X_3)$.
	%\item $(1-p)^{\sum_i (k-r_{i})(r-r_{i})}$ is maximum value of conditional probability that there is no edges between $Y_i\setminus X$ and $(Y\setminus Y_i)\cap X$, given there is no edges between $X_1,X_2$ and $X_3.$
%	\end{itemize}
	%	 We also used that events 
%	 \begin{itemize}
%	 	\item ``there is no edges between $Y_i\setminus X$ and $(Y\setminus Y_i)\cap X$,''
%	 	\item ``there is no edges between $X_1,X_2$ and $X_3$ and each $X_i$ is dominating set in $[n]\setminus(X\cup Y).$''
%	 \end{itemize} 
%    are mutually independent. 

%    Now note that,
%	\begin{equation}\label{S_3^2_bound}
%	\sum_i (k-r_{i})(r-r_{i})=k\sum_i (r-r_{i})-r^2+\sum_i r_{i}^2\geq\text{|FKG|}\geq 2kr-r^2+r^2/3=2r(3k-r)/3.
%	\end{equation}

	The product in the first line of (\ref{S_3_1}) asymptotically equals to ${\sf E}X(k,k,k)=O(k^{3/2})$ due to (\ref{expectation_infinite_asymp}). 
	Moreover, 
	$$
	\left(3nk(1-p)^{2r/3}\right)^{3k-r}\leq
	\left(3\left(\frac{1}{1-p}\right)^k k^2(1-p)^{2k-2r_0/3}\right)^{3k-r}=O(1).
	$$
	Therefore, $S_3^1=O(k^{5/2})$.\\
	
	It remains to bounds $S_3^2$. Applying Lemma \ref{comment} and inequalities $3k-r\leq r_0$, $r_0\geq 16/\ln\frac{1}{1-p}$, we get
	\begin{multline*}\label{S_3_2}
	S^2_3\leq \left[\frac{n!}{k!k!k!(n-3k)!} (1-p)^{3k^2-3} k^3(1-(1-p)^k)^{3(n-6k)}\right]
    n^{3k-r} 3^{3k}k^3(1-p)^{4kr_0-3}=\\
    O\left( k^{4.5} 3^{3k}\left(\frac{k}{(1-p)^k}\right)^{3k-r} (1-p)^{4kr_0}\right)=
    O\left( k^{4.5} 3^{3k} k^{r_0} (1-p)^{3kr_0}\right)=o(1).
	\end{multline*}

	Finally,
	\begin{equation}
	S_3=S_3^1+S_3^2=O(k^{5/2}).
	\label{S3_bound}
	\end{equation}

%	Hence, using (\ref{S_3^2_bound}) and that $r=r_1+r_2+r_3$ with $r_i\geq r_{i,i}\geq k-6r_0$, we get the following bound.
%	\begin{gather*}
%	S_3^1\leq O(k^{3/2})C_{y_1,y_2,y_3}(r)n^{3k-r}\frac{k^{r_1+r_2+r_3}}{r_{1}!r_{2}!r_{3}!}(1-p)^{\frac{2r(3k-r)}{3}}\leq\\
%	O(k^{3/2}) C_{y_1,y_2,y_3}(r)\frac{n^{3k-r} k^r}{(k-6r_0)^{r}}(1-p)^{\frac{2r(3k-r)}{3}}=
%	O(k^{3/2}) C_{y_1,y_2,y_3}(r)\cdot n^{3k-r}(1-p)^{\frac{2r(3k-r)}{3}}\leq O(k^{2.5}).
%	\end{gather*}
	
%	\case{$r>3k-r_0$ and there exists
%	 $s\in\{1,2,3\}$ such that $r_{s,j}<k-6r_0$ for all $j\in\{1,2,3\}$.}
	%$i$ such as for all $ j,$ $r_{i,j}<(k-6r_0)$.}
	%wherein constant $c=6r_0.$}
	%	In (\ref{S_3_2}) factor:
%	\begin{itemize}
	%	\item %$\frac{n!}{k!k!k!(n-3k)!}k^3 n^{3k-r}k^3 3^{3k}$ is the number of pairs of $6$-tuples sharing $r$ elements.
	%	\item $(1-p)^{3k^2}$ is probability that there is no edges between $X_1,X_2$ and $X_3.$
	%	\item $(1-(1-p)^k)^{n-6k}$ is probability that each $X_i$ is dominating set in $[n]\setminus(X\cup Y).$
	%	\item $(1-p)^{4kr_0}$ is is maximum value of conditional probability that there is no edges between $Y_1,Y_2$ and $Y_3$ given that there is no edges between $X_1,X_2$ and $X_3.$(See Proposition \ref{no_edges_bound}
%		(Comment part \ref{comment})).
%	\end{itemize}
%	\begin{gather*}
%	S_3^2=O\left( k^{4.5} 3^{3k}\left(\frac{k}{(1-p)^k}\right)^{3k-r}\cdot (1-p)^{4kr_0}\right)=o(1).
%	\end{gather*}
%where we used previously proved bound on $\frac{n!}{k!k!k!(n-3k)!} (1-p)^{3k^2} k^3(1-(1-p)^k)^{n-6k}$ and that $3k-r<r_0$.
%\end{caseof}
%\section{Proof of Theorem %\ref{mnt:theorem}}
\subsection{Final steps}
\label{sec:P-Z}
%Let us recall Paley-Zygmund inequality:

%\bigskip
%\begin{theorem}
%If $X \geq 0$ is a random variable with finite variance, and if $0 \leq \theta \leq 1$, then
%$$
%\mathbb{P}(X>\theta \mathbb{E}X) \geq(1-\theta)^{2} \frac{(\mathbb{E}X)^{2}}{\mathbb{E}X^{2}}.
%$$
%\end{theorem}

Set $X=X(k,k,k).$ Due to (\ref{expectation_infinite_asymp}), ${\sf E}X\sim\frac{p^3}{\left(\sqrt{2\pi}(1-p)\right)^3}k^{3/2}$.
On the other hand, combining (\ref{second_moment_S}) with bounds (\ref{S1_bound}), (\ref{S2_bound}), (\ref{S3_bound}), we get that there exists $c>0$ such that
${\sf E}X^2=S_1+S_2+S_3<ck^3$.

By Paley-Zygmund inequality (Theorem~\ref{th:PZ}), for $k$ large enough, %with $\theta=0$, we get that our random variable $X=X(k,k,k)$ with positive probability is bounded away from $0$.
$$
\mathbb{P}(X>0)\geq \frac{(\mathbb{E}X)^2}{\mathbb{E} X^2}>
\frac{p^6}{2\pi(1-p)^6c}.
$$
Therefore, $(n_i^{(2)},\,i\in\mathbb{N})$ is the desired sequence.

\section{Acknowledgements}

The first author is a Young Russian Mathematics award winner and would like to thank its sponsors and jury. She is supported  by the program ``Leading Scientific Schools''(grant no. NSh-2540.2020.1).
The second author is supported by by the Ministry of Science and Higher Education of the Russian Federation (Goszadaniye No. 075-00337-20-03), project No. 0714-2020-0005.

\end{document}